\documentclass{siamltex}
\usepackage{graphicx}
\usepackage{amsfonts,latexsym}
\newcommand{\R}{{\mathbb R}}

\newcommand{\bx}{\mbox{\boldmath{$x$}}}
\newcommand{\bb}{\mbox{\boldmath{$b$}}}

\newcommand{\bu}{\mbox{\boldmath{$u$}}}
\newcommand{\bv}{\mbox{\boldmath{$v$}}}

\newcommand{\be}{\mbox{\boldmath{$e$}}}

\newcommand{\sbx}{\mbox{\boldmath{${\scriptstyle x}$}}}

\newcommand{\bl}{\begin{list}{ \ }{
\leftmargin=.325in}}
\newcommand{\el}{\end{list}}

\begin{document}
\date{}
\title{Some matrix nearness problems \\ suggested by Tikhonov regularization}
\author{
Silvia Noschese\thanks{Dipartimento di Matematica ``Guido Castelnuovo'',
SAPIENZA Universit\`a di Roma, P.le A. Moro, 2, I-00185 Roma, Italy. E-mail:
{\tt noschese@mat.uniroma1.it}. Research supported by a grant from
SAPIENZA Universit\`a di Roma.}\and
Lothar Reichel\thanks{Department of
Mathematical Sciences, Kent State University, Kent, OH 44242, USA.
E-mail: {\tt reichel@math.kent.edu}. Research supported in part by NSF grant
DMS-1115385.}
}

\maketitle

\begin{abstract}
The numerical solution of linear discrete ill-posed problems typically requires 
regularization, i.e., replacement of the available ill-conditioned problem by a nearby 
better conditioned one. The most popular regularization methods for problems of small to
moderate size are Tikhonov regularization and truncated singular value decomposition 
(TSVD). By considering matrix nearness problems related to Tikhonov regularization, 
several novel regularization methods are derived. These methods share properties with both
Tikhonov regularization and TSVD, and can give approximate solutions of higher quality 
than either one of these methods.
\end{abstract}

\begin{keywords}
ill-posed problem, Tikhonov regularization, modified Tikhonov regularization, truncated 
singular value decomposition, 
\end{keywords}

\section{Introduction}\label{sec1}
Consider the computation of an approximate solution of the minimization problem
\begin{equation}\label{linsys}
\min_{{\sbx}\in{\R}^n}\|A{\bx}-{\bb}\|,
\end{equation}
where $\|\cdot\|$ denotes the Euclidean vector norm and $A \in {\R}^{m\times n}$ is a 
matrix whose singular values decay smoothly to zero without a significant gap. In 
particular, $A$ may be singular. Minimization problems (\ref{linsys}) with a matrix of
this kind often are referred to as discrete ill-posed problems. They arise, for example, 
from the discretization of linear ill-posed problems, such as Fredholm integral equations 
of the first kind with a smooth kernel. We will for notational simplicity assume that 
$m\geq n$; however, the methods discussed also can be applied when $m<n$. 

The data vector $\bb\in\R^m$ in linear discrete ill-posed problems that arise in science
and engineering typically is contaminated by an (unknown) error $\be\in\R^m$. We will 
refer to the error $\be$ as ``noise.'' Let $\hat{\bb}\in {\R}^m$ denote the 
(unknown) error-free vector associated with ${\bb}$, i.e., 
\begin{equation}\label{rhs}
{\bb}=\hat{\bb}+{\be}.
\end{equation}
The (unknown) linear system of equations with error-free right-hand side,
\begin{equation}\label{linsysnf}
A{\bx}=\hat{\bb},
\end{equation}
is assumed to be consistent; however, we do not require the least-squares 
problem (\ref{linsys}) to be consistent. 

Let $A^\dag$ denote the Moore--Penrose pseudoinverse of $A$. We are interested in computing
an approximation of the solution $\hat{\bx}=A^\dag\hat{\bb}$ of minimal Euclidean norm of 
the error-free linear system (\ref{linsysnf}) by determining an approximate solution of 
the error-contaminated least-squares problem (\ref{linsys}). Note that the solution of 
(\ref{linsys}),
\begin{equation}\label{xbreve}
\bx=A^\dag{\bb}=A^\dag(\hat{\bb}+{\be})=\hat{\bx}+A^\dag{\be},
\end{equation}
typically is dominated by the propagated error $A^\dag{\be}$ and then is meaningless.

Tikhonov regularization, in its simplest form, seeks to determine a useful approximation 
of $\hat{\bx}$ by replacing the minimization problem (\ref{linsys}) by the penalized 
least-squares problem 
\begin{equation}\label{tikhonov}
\min_{{\sbx}\in {\R}^n}\{\|A {\bx}-{\bb}\|^2+\mu^2 \|{\bx}\|^2\}.
\end{equation}
The scalar $\mu>0$ is a regularization parameter. We are interested in developing 
modifications of this minimization problem by considering certain matrix nearness 
problems.

Solving (\ref{tikhonov}) requires both the determination of a suitable value of $\mu>0$ 
and the computation of the associated solution 
\begin{equation}\label{tiksol}
{\bx}_\mu=(A^TA+\mu^2 I)^{-1}A^T{\bb}
\end{equation}
of (\ref{tikhonov}). Throughout this paper the superscript $^T$ denotes transposition and
$I$ is the identity matrix of appropriate order. We will assume that a bound for the norm 
of the error-vector ${\be}$ is known. Then $\mu$ can be determined with the aid of the 
discrepancy principle; see below for details. 

Another common regularization method for (\ref{linsys}) is truncated singular value
decomposition (TSVD). In this method the $n-k$ smallest singular values of $A$ are set to
zero and the minimal-norm solution of the resulting least-squares problem is computed. The
truncation index $k$ is a regularization parameter, which can be determined, e.g., with 
the discrepancy principle.

The TSVD method generally only dampens high frequencies in the computed solution, while 
Tikhonov regularization (\ref{tikhonov}) dampens all frequencies. A modification of the
Tikhonov minimization problem (\ref{tikhonov}) that generally only dampens high 
frequencies has been described in \cite{FR}. This modification can be derived as the 
solution of a matrix nearness problem. It is the purpose of this paper to describe
several matrix nearness problems that suggest modifications of the Tikhonov minimization
problem (\ref{tikhonov}). Some of these modifications perform particularly well for 
problems (\ref{linsys}) in which the vector $\bb$ is contaminated by colored noise 
dominated by high-frequency components.

This paper is organized as follows. Section \ref{sec2} reviews TSVD and Tikhonov 
regularization, as well as the modified Tikhonov regularization method described in 
\cite{FR}, and introduces new regularization methods suggested by certain matrix nearness
problems. Section \ref{sec3} presents a few computed examples, and Section \ref{sec4} 
contains concluding remarks 
and discusses some extensions. In particular, the discussion of methods in this paper 
assumes the singular value decomposition (SVD) of the matrix $A$ to be available. However,
it is impractical to compute the SVD of large matrices. We comment in Section \ref{sec4} 
on how the methods of this paper can be applied to the solution of large-scale 
least-squares problems (\ref{linsys}).

\section{Old and new regularization methods}\label{sec2}
We first describe the SVD of $A$, then review 
regularization by the TSVD and Tikhonov methods, and finally describe several 
modifications of the Tikhonov minimization problem (\ref{tikhonov}). The SVD of $A$ is a
factorization of the form
\begin{equation}\label{svd}
A=U\Sigma V^T,
\end{equation}
where $U=[{\bu}_1,{\bu}_2,\ldots,{\bu}_m]\in{\R}^{m\times m}$ and 
$V=[{\bv}_1,{\bv}_2,\ldots,{\bv}_n]\in{\R}^{n\times n}$ are orthogonal 
matrices, the superscript $^T$ denotes transposition, and 
\[
\Sigma={\rm diag}[\sigma_1,\sigma_2,\ldots,\sigma_n]\in{\R}^{m\times n}
\]
is a (possibly rectangular) diagonal matrix, whose diagonal entries $\sigma_j\geq 0$ are 
the singular values of $A$. They are ordered according to 
$\sigma_1\geq\sigma_2\geq\ldots\geq\sigma_n$. 

Let $A$ be of rank $\ell\geq 1$. Then (\ref{svd}) can be expressed as 
\begin{equation}\label{svdsum}
A=\sum_{j=1}^\ell \sigma_j{\bu}_j{\bv}_j^T
\end{equation}
with $\sigma_\ell>0$. When the matrix $A$ stems from the discretization of a 
compact operator, such as a Fredholm integral equation of the first kind with a smooth 
kernel, the vectors ${\bv}_j$ and ${\bu}_j$ represent discretizations of singular 
functions that are defined on the domains of the integral operator and its adjoint, 
respectively. These singular functions typically oscillate more with increasing index. The 
representation (\ref{svdsum}) then is a decomposition of $A$ into rank-one matrices 
${\bu}_j{\bv}_j^T$ that are discretizations of products of singular functions that 
oscillate more with increasing index $j$.

\subsection{Regularization by TSVD}
\label{sub2.1}
The Moore--Penrose pseudoinverse of $A$ is given by
\[
A^\dag=\sum_{j=1}^\ell \sigma_j^{-1}{\bv}_j{\bu}_j^T.
\]
The difficulty of solving (\ref{linsys}) without regularization stems from the fact that 
the matrix $A$ has ``tiny'' positive singular values and the computation of the solution 
(\ref{xbreve}) of (\ref{linsys}) involves division by these singular values. This results 
in severe propagation of the error ${\be}$ in ${\bb}$ and of round-off errors introduced 
during the calculations of the computed approximate solution of (\ref{linsys}). 

Regularization by the TSVD method overcomes this difficulty by ignoring the tiny positive 
singular values of $A$. Introduce, for $1\leq k\leq\ell$, the rank-$k$ approximation of 
$A$,
\[
A_k=\sum_{j=1}^k \sigma_j{\bu}_j{\bv}_j^T
\]
with Moore--Penrose pseudoinverse 
\[
A_k^\dag=\sum_{j=1}^k \sigma_j^{-1}{\bv}_j{\bu}_j^T.
\]
The TSVD method yields approximate solutions of (\ref{linsys}) of the form
\begin{equation}\label{TSVDk}
{\bx}_k=A_k^\dag{\bb}=\sum_{j=1}^k \frac{{\bu}_j^T{\bb}}{\sigma_j} {\bv}_j,
\qquad k=1,2,\ldots,\ell.
\end{equation}
It is convenient to use the transformed quantities
\[
\widetilde{\bx}_k=V^T{\bx}_k,\qquad 
\widetilde{\bb}=[\widetilde{b}_1,\widetilde{b}_2,\ldots,\widetilde{b}_m]^T=
U^T{\bb}
\]
in the computations. Thus, we compute
\begin{equation}\label{yk}
\widetilde{\bx}_k=
\left[\frac{\widetilde{b}_1}{\sigma_1},\frac{\widetilde{b}_2}{\sigma_2},\ldots,
\frac{\widetilde{b}_k}{\sigma_k},0,\ldots,0\right]^T
\end{equation}
for a suitable value of $1\leq k\leq\ell$ and then determine the approximate solution 
${\bx}_k=V\widetilde{\bx}_k$ of (\ref{linsys}). 

Let a bound for the norm of the error
\[
\|{\be}\|\leq\varepsilon
\]
in $\bb$ be available. We then can determine a suitable truncation index $k$ by the 
discrepancy principle, i.e., we choose $k$ as small as possible so that 
\begin{equation}\label{discrprinc}
\|A{\bx}_k-{\bb}\|\leq\eta\varepsilon,
\end{equation}
where $\eta\geq 1$ is a user-specified constant independent of $\varepsilon$. Thus, the 
truncation index $k=k_{\varepsilon}$ depends on $\varepsilon$ and generally increases as
$\varepsilon$ decreases. A proof of the convergence of ${\bx}_{k_\varepsilon}$ to 
$\hat{\bx}$ as $\varepsilon\searrow 0$ in a Hilbert space setting is presented in 
\cite{EHN}. It requires $\eta>1$ in (\ref{discrprinc}). In actual computations, we use the
representation
\[
\|A{\bx}_k-{\bb}\|^2=\sum_{j=k+1}^m \widetilde{b}_j^2
\]
to determine $k_\varepsilon$ from (\ref{discrprinc}). Further details on regularization by
the TSVD method can be found in, e.g., \cite{EHN,Ha1}.

\subsection{Standard Tikhonov regularization}\label{sub2.2}
Substituting (\ref{svd}), $\widetilde{\bx}=V^T{\bx}$, and $\widetilde{\bb}=U^T{\bb}$ into
(\ref{tikhonov}) yields the penalized least-squares problem
\[
\min_{\widetilde{\sbx}\in{\R}^n}\{\|\Sigma\widetilde{\bx}-\widetilde{\bb}\|^2
+\mu^2\|\widetilde{\bx}\|^2\}
\]
with solution 
\begin{equation}\label{ymu}
\widetilde{\bx}_\mu=(\Sigma^T\Sigma+\mu^2 I)^{-1}\Sigma^T\widetilde{\bb}
\end{equation}
for any $\mu>0$. The associated solution of (\ref{tikhonov}) is given by 
${\bx}_\mu=V\widetilde{\bx}_\mu$. It satisfies 
\begin{equation}\label{tikh}
(A^TA+\mu^2 I){\bx}_\mu=A^T{\bb}.
\end{equation}

The discrepancy principle prescribes that the regularization parameter $\mu>0$
be determined so that
\begin{equation}\label{xx}
\|A{\bx}_\mu-{\bb}\|=\eta\varepsilon,
\end{equation}
or, equivalently, so that 
\begin{equation}\label{discrprinc2}
\|\Sigma\widetilde{\bx}_\mu-\widetilde{\bb}\|=\eta\varepsilon,
\end{equation}
where $\eta\geq 1$ is a user-chosen constant independent of $\varepsilon$. This nonlinear 
equation for $\mu$ can be solved, e.g., by Newton's method. Generally, $\mu$ decreases 
with $\varepsilon$. A proof of the convergence ${\bx}_\mu\rightarrow\hat{\bx}$ as 
$\varepsilon\searrow 0$ is provided in \cite{EHN}. The proof is in a Hilbert space 
setting and requires that $\eta>1$ in (\ref{xx}). All methods discussed in Subsections
\ref{sub2.3} and \ref{sub2.5} use the value of $\mu$ determined by (\ref{xx}), i.e., 
$\mu>0$ is for all methods chosen so that the solution ${\bx}_\mu$ of (\ref{tikhonov}) 
satisfies (\ref{xx}).

\subsection{Modified Tikhonov regularization}\label{sub2.3}
It follows from (\ref{ymu}) that Tikhonov regularization with $\mu>0$ dampens all solution
components ${\bv}_j$ of ${\bx}_\mu$. On the other hand, TSVD does not dampen any solution 
component that is not set to zero; cf. (\ref{yk}). It is well known that Tikhonov 
regularization may oversmooth the computed solution when the regularization parameter is 
determined by the discrepancy principle; see  Hansen \cite[\S 7.2]{Ha1}. A more recent 
discussion on the oversmoothing of the solution (\ref{tiksol}) obtained with Tikhonov
regularization is provided by Klann and Ramlau \cite{KR}. 

In order to reduce the oversmoothing, it was suggested in \cite{FR} that the minimization 
problem (\ref{tikhonov}) be replaced by 
\begin{equation}\label{tikh2}
\min_{{\sbx}\in {\R}^n}\{\|A {\bx}-{\bb}\|^2+ \|L_\mu{\bx}\|^2\},
\end{equation}
where 
\begin{equation}\label{Lnew}
L_\mu=D_\mu V^T
\end{equation}
and 
\[
D_\mu^2={\rm diag}\left[\max\{\mu^2-\sigma_1^2,0\},\max\{\mu^2-\sigma_2^2,0\},
\ldots,\max\{\mu^2-\sigma_n^2,0\}\right].
\]
Thus, the elements of $D_\mu$, and therefore of $L_\mu$, are nonlinear functions of 
$\mu\geq 0$. Analogously to (\ref{ymu}), one has
\begin{equation}\label{ymu2}
\widetilde{\bx}_\mu=(\Sigma^T\Sigma+D_\mu^2)^{-1}\Sigma^T\widetilde{\bb}.
\end{equation}

We determine $\mu\geq 0$ so that the solution (\ref{tiksol}) of standard Tikhonov 
regularization (\ref{tikhonov}) satisfies the discrepancy principle (\ref{xx}). If
$\mu\geq\sigma_1$, then 
\[
\Sigma^T\Sigma+D_\mu^2=\mu^2 I.
\]
If, instead, $0\leq\mu<\sigma_1$, then there is $1\leq k\leq n$ such that 
$\sigma_k>\mu\geq\sigma_{k+1}$, where we define $\sigma_{n+1}=0$ when $k=n$. These
values of $\mu$ and $k$ yield
\[
\Sigma^T\Sigma+D_\mu^2=
{\rm diag}\left[\sigma_1^2,\sigma_2^2,\ldots,\sigma_k^2,\mu^2,\ldots,\mu^2
\right]\in{\R}^{n\times n}.
\]

We will in the remainder of this section assume that $k\geq 1$. When $\mu>0$, the above 
matrix is positive definite and the solution (\ref{ymu2}) exists and is unique. The 
corresponding approximate solution of (\ref{linsys}) is given by 
${\bx}_\mu=V\widetilde{\bx}_\mu$ and satisfies 
\begin{equation}\label{X.X}
(A^TA+L_\mu^TL_\mu){\bx}=A^T{\bb}.
\end{equation}

To avoid severe propagation of the error ${\be}$ in ${\bb}$ into the solution of 
(\ref{X.X}), the matrix $A^TA+L_\mu^TL_\mu$ should not be too ill-conditioned. This can be 
achieved by letting $\mu>0$ be sufficiently large. We measure the conditioning of a matrix 
by its spectral condition number $\kappa_2$, which is defined as the ratio of the 
largest and smallest positive singular values of the matrix. For instance,
\begin{eqnarray}
\nonumber
\kappa_2(A_k)&=&\frac{\sigma_1}{\sigma_k},\qquad 1\leq k\leq\ell,\\
\label{condstd}
\kappa_2(A^TA+\mu^2 I)&=&\frac{\sigma_1^2+\mu^2}{\sigma_n^2+\mu^2},\\ 
\label{condmod}
\kappa_2(A^TA+L_\mu^TL_\mu)&=&\frac{\sigma_1^2}{\mu^2},\qquad \sigma_n\leq\mu<\sigma_1.
\end{eqnarray}

It is desirable that the matrix $L_\mu^TL_\mu$ be of small norm so that 
equation (\ref{X.X}) is fairly close to the normal equations $A^TA\bx=A^T\bb$ 
associated with (\ref{linsys}), because this may help us determine an accurate 
approximation of $\hat{\bx}$. Indeed, the matrix $L_\mu^T L_\mu$ can be shown to be the
closest matrix to $A^TA$ in the Frobenius norm with the property that its smallest 
singular value is $\mu^2$; see \cite[Theorem 2.1 and Corollary 2.2]{FR}. We recall that
the Frobenius norm of a matrix $M\in{\R}^{n\times n}$ is given by 
$\|M\|_F=\sqrt{{\rm trace}(M^TM)}$.

\subsection{Filter factors}
Properties of regularization methods can be studied with the aid of filter factors; see,
e.g., Hansen \cite{Ha1} and Donatelli and Serra--Capizzano \cite{DSC} for illustrations.
The unregularized solution (\ref{xbreve}) can be expressed as
\[
{\bx}=\sum_{j=1}^\ell \frac{{\bu}_j^T{\bb}}{\sigma_j} {\bv}_j.
\]
The filter factors show how the components are modified by a regularization method. For 
instance, we can express the TSVD solution (\ref{TSVDk}) as 
\[
{\bx}_k=\sum_{j=1}^\ell \varphi_{k,j}^{({\rm TSVD})}
\frac{{\bu}_j^T{\bb}}{\sigma_j} {\bv}_j
\]
with the filter factors
\[
\varphi_{k,j}^{({\rm TSVD})}=\left\{\begin{array}{cc} 1,~~&~~1\leq j\leq k, \\
                       0,~~&~~k< j\leq\ell. \end{array}\right.
\]
Similarly, the Tikhonov solution of (\ref{tikh}) can be written as 
\[
{\bx}_\mu=\sum_{j=1}^\ell \varphi_{\mu,j}^{({\rm Tikhonov})}
\frac{{\bu}_j^T{\bb}}{\sigma_j} {\bv}_j
\]
with the filter factors
\[
\varphi_{\mu,j}^{({\rm Tikhonov})}=\frac{\sigma_j^2}{\sigma_j^2+\mu^2}, 
\qquad 1\leq j\leq\ell.
\]

Let $\mu>0$ and assume that $k$ is such that $\sigma_k>\mu\geq\sigma_{k+1}$, where we
define $\sigma_{n+1}=0$ if $k=n$. The solution of the modified Tikhonov regularization
method (\ref{tikh2}) can be expressed as
\[
{\bx}_\mu=\sum_{j=1}^\ell \varphi_{\mu,j} \frac{{\bu}_j^T{\bb}}{\sigma_j} {\bv}_j
\]
with the filter factors
\[
\varphi_{\mu,j}^{}=\left\{\begin{array}{cc} 1,~~&~~1\leq j\leq k, \\
       \displaystyle{\frac{\sigma_j^2}{\mu^2}},~~&~~k< j\leq\ell.
		       \end{array}\right.
\]
Thus, these filter factors are the same as $\varphi_{k,j}^{({\rm TSVD})}$ for 
$1\leq j\leq k$, and close to $\varphi_{\mu,j}^{({\rm Tikhonov})}$ for $k< j\leq\ell$.

\subsection{New modified Tikhonov regularization methods}\label{sub2.5}
This section derives new modifications of Tikhonov regularization (\ref{tikhonov}) by
focusing on condition numbers. 
For all methods of this subsection, we determine $\mu\geq 0$ similarly as in Subsection 
\ref{sub2.3}, i.e., so that the solution (\ref{tiksol}) of (\ref{tikhonov}) satisfies 
(\ref{xx}). Then $k$ is chosen as a function of $\mu$ as described.

\begin{proposition}\label{prop0}
Let $L_\mu$ be defined by (\ref{Lnew}) and assume that $\sigma_n \leq \mu \leq \sigma_1$.
Then
\begin{equation}\label{cond2.9}
\max\{\kappa_2(A^TA+L_\mu^TL_\mu),\kappa_2(A^TA+\mu^2 I)\}\leq \kappa_2(A^TA).
\end{equation}
Moreover,
\begin{equation}\label{cond3} 
\kappa_2(A^TA+L_\mu^TL_\mu) \leq \kappa_2(A^TA+\mu^2 I) \Leftrightarrow 
\mu^2\geq\sigma_1\sigma_n.
\end{equation}
\end{proposition}

\begin{proof}
The proofs of the inequalities (\ref{cond2.9}) and (\ref{cond3}) follow from 
(\ref{condstd}) and (\ref{condmod}). The requirement on $\mu^2$ in (\ref{cond3}) typically
is satisfied for linear discrete ill-posed problems that arise in applications.
\end{proof}

We discuss Tikhonov regularization for several regularization matrices that are 
modifications of  $\mu I$ and yield condition numbers of the associated normal equations 
that are smaller than the condition number (\ref{condstd}) of the matrix $A^TA+\mu^2 I$.
We first consider the regularization matrix
\begin{equation}\label{Lnew2}
L_{\mu,k}=D_{\mu,k} V^T
\end{equation}
with 
\[
D_{\mu,k}={\rm diag}\left[0,0,
\ldots,0,\overbrace{\mu,\ldots,\mu}^{n-k}\right].
\]
Given $\mu\geq 0$, the index $k=k_\mu$ is chosen so that the diagonal entries of 
\[
\Sigma^T\Sigma+D_{\mu,k}^2={\rm diag}\left[\sigma_1^2,\sigma_2^2,\ldots,\sigma_k^2,
\sigma_{k+1}^2+\mu^2,\ldots,\sigma_{n}^2+\mu^2 \right],
\]
are non-increasing when the column index increases. Thus, the regularization matrix 
(\ref{Lnew2}) leaves the largest $k$ eigenvalues of $A^TA$ invariant and shifts the
remaining ones.

\begin{proposition}\label{prop3.1}
Let $L_\mu$ and $L_{\mu,k}$ be defined by (\ref{Lnew}) and (\ref{Lnew2}), respectively,
and assume that $k=k_\mu$ in $L_{\mu,k}$ is chosen as described above. Then
\[
\kappa_2(A^TA+L_{\mu,k}^TL_{\mu,k})=\frac{\sigma_1^2}{\sigma_n^2+\mu^2}.
\]
Therefore 
\begin{equation}\label{cond5}
\kappa_2(A^TA+L_{\mu,k}^TL_{\mu,k}) \leq \kappa_2(A^TA+\mu^2 I)\Leftrightarrow \mu^2\ne 0
\end{equation}
and
\begin{equation}\label{cond6}
\kappa_2(A^TA+L_{\mu,k}^TL_{\mu,k}) \leq \kappa_2(A^TA+L_\mu^TL_\mu),  
\end{equation}
where the latter inequality is strict  if and only if $A$ is of full rank. Moreover,
for $k\geq 1$,
\begin{equation}\label{prp3}
\|L_{\mu,k}\|_F<\|\mu I\|_F.
\end{equation}
\end{proposition}

\begin{proof}
The proofs of (\ref{cond5}) and (\ref{cond6}) are immediate. The inequality (\ref{prp3}) 
follows from the observation that 
\[
\|L_{\mu,k}\|_F^2=\|D_{\mu,k}\|_F^2=(n-k)\mu^2.
\]
\end{proof}

The filter factors for Tikhonov regularization with the regularization matrix 
(\ref{Lnew2}) are given by
\begin{equation}\label{phimukj}
\varphi_{\mu,k,j}=\left\{\begin{array}{cc} 1,~~&~~1\leq j\leq k, \\
       \displaystyle{\frac{\sigma_j^2}{\sigma_j^2+\mu^2}},~~&~~k< j\leq\ell.
		       \end{array}\right.
\end{equation}
Thus, these filter factors are the same as $\varphi_{k,j}^{({\rm TSVD})}$ for 
$1\leq j\leq k$, and the same as $\varphi_{\mu,j}^{({\rm Tikhonov})}$ for $k< j\leq\ell$.
However, the discrepancy principle applied to TSVD, cf. (\ref{discrprinc}), may yield a 
different value of $k$.

We are lead to an alternative to the regularization matrix (\ref{Lnew2}) when we instead
of shifting the smallest eigenvalues of $A^TA$ ignore them. Define the 
regularization matrix
\begin{equation}\label{Lnew3}
L_k=D_k V^T
\end{equation}
with 
\[
D_k^2={\rm diag}\left[0,0,
\ldots,0,-\sigma_{k+1}^2, \dots, -\sigma_{n}^2\right].
\]
Then
\[
\Sigma^T\Sigma+D_k^2=\Sigma_k^T\Sigma_k=
{\rm diag}\left[\sigma_1^2,\sigma_2^2,\ldots,\sigma_k^2,0,0,\ldots,0\right].
\] 

\begin{proposition}\label{prop3.2}
Let the regularization matrices $L_\mu$, $L_{\mu,k}$, and $L_k$ be defined by 
(\ref{Lnew}), (\ref{Lnew2}), and (\ref{Lnew3}), respectively. Then
\[
\kappa_2(A^TA+L_k^TL_k)=\kappa_2(A_{k}^TA_{k})=\frac{\sigma_1^2}{\sigma_k^2}.
\]
Therefore,
\begin{equation}\label{cond7}
\kappa_2(A^TA+L_{k}^TL_{k}) \leq \kappa_2(A^TA+L_{\mu}^TL_{\mu}) \Leftrightarrow 
\mu \leq \sigma_k
\end{equation}
and
\begin{equation}\label{cond8}
\kappa_2(A^TA+L_{k}^TL_{k}) \leq \kappa_2(A^TA+L_{\mu,k}^TL_{\mu,k}) \Leftrightarrow
\mu \leq \sqrt{\sigma_k^2-\sigma_n^2}.
\end{equation}
Moreover, if  $\sigma_{k+1}\leq\mu<\sigma_k$, then
\begin{equation}\label{cut}
\|L_k\|_F\leq \|L_{\mu,k}\|_F.
\end{equation}
\end{proposition}

\begin{proof}
The inequalities (\ref{cond7}) and (\ref{cond8}) are straightforward. Property (\ref{cut})
follows from 
\[
\|L_k\|_F^2=\|D_k\|_F^2=\sum_{\sigma_j^2\leq\mu^2}\sigma_j^2\leq
(n-k)\mu^2=\|D_{\mu,k}\|_F^2=\|L_{\mu,k}\|_F^2.
\]
\end{proof}

The filter factors for Tikhonov regularization with the regularization matrix 
(\ref{Lnew3}) are the same as $\varphi_{k,j}^{({\rm TSVD})}$.

The observations at the end of Subsection \ref{sub2.3} suggest that we seek to determine 
regularization matrices that give normal equations with the same condition number as 
$A^T A +\mu^2 I$ but have smaller Frobenius norm than $\mu I$. Introduce the regularization 
matrix
\begin{equation}\label{Lnew4}
\widetilde L_\mu=\widetilde D_\mu V^T
\end{equation}
with 
\[
\widetilde D_\mu^2=\frac{\mu^2}{\sigma_1^2+\mu^2}\,{\rm diag}\left[0,\sigma_1^2-\sigma_2^2,
\ldots,\sigma_1^2-\sigma_n^2\right].
\]
Then
\begin{equation}\label{D1}
\Sigma^T\Sigma+\widetilde D_\mu^2={\rm diag}\left[\sigma_1^2,
\frac{\sigma_1^2}{\sigma_1^2+\mu^2}(\sigma_2^2+\mu^2),\ldots,
\frac{\sigma_1^2}{\sigma_1^2+\mu^2}(\sigma_n^2+\mu^2) \right].
\end{equation}

\begin{proposition}\label{prop3.3}
Let $\widetilde L_\mu$ be given by (\ref{Lnew4}). Then
\begin{equation}\label{cd1}
\kappa_2(A^TA+\widetilde L_\mu^T \widetilde L_\mu)=\kappa_2(A^TA+\mu^2 I)
\end{equation}
and
\begin{equation}\label{cd2}
\|\widetilde L_\mu\|_F<\|\mu I\|_F.
\end{equation}
\end{proposition}

\begin{proof}
The equality (\ref{cd1}) follows from (\ref{D1}). The inequality (\ref{cd2}) is a 
consequence of 
\[
\|\widetilde L_\mu\|_F^2=\|\widetilde D_\mu\|_F^2=\frac{\mu^2}{\sigma_1^2+\mu^2}
\sum_{i=2}^{n}(\sigma_1^2-\sigma_i^2) <(n-1)\mu^2<\|\mu I\|_F^2.
\]
\end{proof}

The filter factors for Tikhonov regularization with the regularization matrix 
(\ref{Lnew4}) are given by
\[
\widetilde\varphi_{\mu,j}=
\frac{\sigma_j^2(\sigma_1^2+\mu^2)}{\sigma_1^2(\sigma_j^2+\mu^2)},
\qquad 1\leq j\leq\ell.
\]
Thus, these filter factors are the same as $\varphi_{k,j}^{({\rm TSVD})}$
for $j=1$, and close to $\varphi_{\mu,j}^{({\rm Tikhonov})}$ for $1<j\leq\ell$. 
Specifically, 
\[
\widetilde \varphi_{\mu,j}=\frac{(\sigma_1^2+\mu^2)}
{\sigma_1^2}\varphi_{\mu,j}^{({\rm Tikhonov})},\qquad 1<j\leq\ell.
\]

Another regularization matrix that also yields regularized normal equations with the same
spectral condition number as $A^T A +\mu^2 I$ is given by 
\begin{equation}\label{Lnew5}
\widetilde L_{\mu,k}=\widetilde D_{\mu,k} V^T
\end{equation}
with 
\[
\widetilde D_{\mu,k}^2=\frac{\mu^2}{\sigma_1^2+\mu^2}\,{\rm diag}\left[0,\dots, 0,\sigma_1^2-\sigma_{k+1}^2,
\ldots,\sigma_1^2-\sigma_n^2\right].
\]
Then
\begin{equation}\label{eval5}
~~~~~~~~~\Sigma^T\Sigma+\widetilde D_{\mu,k}^2=
{\rm diag}\left[\sigma_1^2,\ldots, \sigma_k^2, \frac{\sigma_1^2}{\sigma_1^2+\mu^2}(\sigma_{k+1}^2+\mu^2),\ldots,\frac{\sigma_1^2}{\sigma_1^2+\mu^2}(\sigma_n^2+\mu^2)
\right].
\end{equation}
The index $k=k_\mu$ is chosen so that the diagonal entries of 
$\Sigma^T\Sigma+\widetilde D_{\mu,k}^2$ are nonincreasing. The following results are 
analogous to those of Proposition \ref{prop3.3}.

\begin{proposition}\label{prop3.4}
Let the matrix $\widetilde L_{\mu,k}$ be defined by (\ref{Lnew5}) with the index $k=k_\mu$
chosen as indicated above. Then
\begin{equation}\label{cd3}
\kappa_2(A^TA+\widetilde L_{\mu,k}^T \widetilde L_{\mu,k})=\kappa_2(A^TA+\mu^2 I)
\end{equation}
and
\begin{equation}\label{cd4}
\|\widetilde L_{\mu,k}\|_F^2<(n-k)\mu^2<\|\mu I\|_F^2.
\end{equation}
\end{proposition}

\begin{proof}
Property (\ref{cd3}) is a consequence of (\ref{eval5}), and (\ref{cd4}) follows from the
choice of $k$, i.e., $\sigma_{k+1}\leq\mu<\sigma_k$.

Indeed, the squared Frobenius norm of the regularization matrix defined by $\widetilde L_{\mu,k}$ in (\ref{Lnew5}) is less than or equal to 
that of the one defined by $\widetilde L_{\mu}$ in (\ref{Lnew4}), i.e. 
\[
\|\widetilde L_{\mu,k}\|_F^2=\|\widetilde D_{\mu,k}\|_F^2=
\frac{\mu^2}{\sigma_1^2+\mu^2}\sum_{i=k+1}^{n}(\sigma_1^2-\sigma_i^2)<(n-k)\mu^2<
\|\mu I\|_F^2.
\]
\end{proof}

We next compare the regularization matrices (\ref{Lnew}) and (\ref{Lnew5}). 

\begin{proposition}\label{pro3.1}
Let $L_{\mu}$ and  $\widetilde L_{\mu,k}$ be given by (\ref{Lnew}) and (\ref{Lnew5}), 
respectively. Assume that $k$ is such that 
\begin{equation}\label{k}
\sigma_k>\frac{\mu^2}{\sigma_{1}}\geq \sigma_{k+1}.
\end{equation} 
Then
\[
\|\widetilde L_{\mu,k}\|_F\leq\| L_{\mu}\|_F.
\]
\end{proposition}

\begin{proof}
For any $j>k$, one has $\sigma_1\sigma_j\leq \mu^2$. Therefore, 
\[
\frac{\mu^2}{\sigma_1^2+\mu^2}(\sigma_1^2-\sigma_j^2)\leq\mu^2-\sigma_j^2,
\]
and it follows that 
\[
\|\widetilde L_{\mu,k}\|_F^2=\frac{\mu^2}{\sigma_1^2+\mu^2}
\sum_{j=k+1}^{n}(\sigma_1^2-\sigma_j^2)\leq \sum_{j=k+1}^n(\mu^2-\sigma_j^2).
\]
Assuming $\mu< \sigma_1$, so that $\mu^2< \mu\sigma_1$, we obtain 
\[
\sum_{j=k+1}^n(\mu^2-\sigma_j^2) \leq \sum_{\sigma_j^2<\mu^2}(\mu^2-\sigma_j^2),
\]
which concludes the proof.
\end{proof}

Note that the parameter $k$ such that (\ref{k}) is satisfied may differ from the parameter 
${\tilde k}$ such that $\sigma_{\tilde k} > \mu\geq \sigma_{\tilde k +1}$. Specifically,
$k\geq{\tilde k}$.

We also can establish the relations
\[
\|\widetilde L_{\mu,k}\|_F < \| L_{\mu,k}\|_F, \qquad
\|\widetilde L_{\mu,k}\|_F \leq \|\widetilde L_{\mu}\|_F,
\]
where the latter inequality is strict if $\sigma_1>\sigma_k$. Thus, the regularization 
matrix $\widetilde L_{\mu,k}$ yields normal equations with the same condition number
as the regularization matrix $\mu I$, but is of smaller norm than this and several other
regularization matrices considered. We therefore expect $\widetilde L_{\mu,k}$ to often 
yield more accurate approximations of the desired solution $\hat{\bx}$ than the other 
regularization matrices discussed above. That this is, indeed, the case is illustrated in 
Section \ref{sec3}.

The filter factors for Tikhonov regularization with the regularization matrix 
(\ref{Lnew5}) are given by
\[
\widetilde\varphi_{\mu,k,j}=\left\{\begin{array}{cc} 1,~~&~~1\leq j\leq k, \\
  \displaystyle{\frac{\sigma_j^2(\sigma_1^2+\mu^2)}
  {\sigma_1^2(\sigma_j^2+\mu^2)}},~~&~~k< j\leq\ell,\end{array}\right.
\]
i.e., they are same as $\varphi_{k,j}^{({\rm TSVD})}$ for $1\leq j\leq k$, and are close
to $\varphi_{\mu,j}^{({\rm Tikhonov})}$ for $k< j\leq\ell$.

The above analysis suggests that we introduce a parameter $\theta$ that allows us to 
interpolate between the regularization matrices (\ref{Lnew2}) and (\ref{Lnew5}). Thus,
define for $0\leq \theta \leq 1$ the regularization matrices
\begin{equation}\label{Ltheta}
L_{\mu,k}(\theta)=D_{\mu,k}(\theta)V^T
\end{equation}
with
\[
D_{\mu,k}^2(\theta)=\frac{\mu^2}{\sigma_1^2+\theta\mu^2}\,{\rm diag}\left[0,\dots, 0,\sigma_1^2-\theta\sigma_{k+1}^2,
\ldots,\sigma_1^2-\theta\sigma_n^2\right].
\]
Then 
\[
\Sigma^T\Sigma+ D_{\mu,k}^2(\theta)=
{\rm diag}\left[\sigma_1^2,\ldots, \sigma_k^2, \frac{\sigma_1^2}{\sigma_1^2+\theta\mu^2}(\sigma_{k+1}^2+\mu^2),\ldots,\frac{\sigma_1^2}{\sigma_1^2+\theta\mu^2}(\sigma_n^2+\mu^2)
\right]
\]
from which it follows that
\begin{eqnarray*}
\kappa_2(A^TA+L_{\mu,k}(\theta)^TL_{\mu,k}(\theta))&=&
 (1-\theta)\kappa_2(A^TA+L_{\mu,k}^TL_{\mu,k})+
 \theta \kappa_2(A^TA+\widetilde L_{\mu,k}^T\widetilde L_{\mu,k})\\
 &=& \frac{\sigma_1^2+\theta \mu^2}{\sigma_n^2+\mu^2}.
\end{eqnarray*}
Moreover,
\[
\|L_{\mu,k}(\theta)\|_F^2=\frac{\mu^2}{\sigma_1^2+\theta\mu^2}
\sum_{i=k+1} ^{n}(\sigma_1^2-\theta\sigma_i^2).
\]
Hence, the norm $\| L_{\mu,k}(\theta)\|_F^2$ is a nonincreasing function of $\theta$, whereas the condition number 
$\kappa_2(A^TA+L_{\mu,k}(\theta)^TL_{\mu,k}(\theta))$ is an increasing function of 
$\theta$.

The filter factors for Tikhonov regularization with the regularization matrix 
(\ref{Ltheta}) are given by
\[
\varphi_{\mu,k,j}(\theta)=(1-\theta)\varphi_{\mu,k,j}+\theta\widetilde \varphi_{\mu,k,j}=
\left\{\begin{array}{cc} 1,~~&~~1\leq j\leq k, \\
\displaystyle{\frac{\sigma_j^2(\sigma_1^2+\theta\mu^2)}{\sigma_1^2(\sigma_j^2+\mu^2)}},
~~&~~k< j\leq\ell, 
\end{array}\right.
\]
where $\varphi_{\mu,k,j}$ is defined by (\ref{phimukj}). Thus, the filter factors
$\varphi_{\mu,k,j}(\theta)$ agree with $\varphi_{k,j}^{({\rm TSVD})}$ for $1\leq j\leq k$,
and are close to $\varphi_{\mu,j}^{({\rm Tikhonov})}$ for $k<j\leq\ell$.

Numerical examples in the following section show the regularization matrices 
$L_{\mu,k}(1)=\widetilde L_{\mu,k}$ and $L_{\mu,k}(0)= L_{\mu,k}$ to yield the most
accurate approximations of $\hat{\bx}$. The former matrix has the smallest Frobenius norm
and the latter yields normal equations for Tikhonov regularization with the smallest 
condition number.

\section{Computed examples}\label{sec3}
The calculations of this section were carried out using MATLAB with relative accuracy  
$2.2\cdot 10^{-16}$. Most of the examples are obtained by discretizing Fredholm integral 
equations of the first kind 
\begin{equation}\label{fredholm}
    \int_a^b\! h(s,t) x(t) \,dt = g(s), \qquad c\leq s\leq d,
\end{equation}
with a smooth kernel $h$. The discretizations are carried out by Galerkin or 
Nystr\"om methods and yield linear discrete ill-posed problems (\ref{linsys}). MATLAB 
functions in Regularization Tools \cite{Ha2} determine discretizations 
$A\in{\R}^{m\times n}$ of the integral operators and scaled discrete approximations 
$\hat{\bx}\in{\R}^n$ of the solution $x$ of (\ref{fredholm}). In all examples, we let
$m=n=200$. The performance of the regularization matrices discussed in this paper is 
illustrated when the error $\be$ in $\bb$ is white Gaussian noise or colored noise. We 
begin with the former.

\subsection{Tests with white noise}
In the experiments of this subsection the error vector ${\be}\in{\R}^m$ has normally 
distributed random entries with zero mean. The vector is scaled to yield a specified noise
level $\|{\be}\|/\|\hat{\bb}\|$ and added to the error-free data vector 
$\hat{\bb}:=A\hat{\bx}$ to obtain the vector ${\bb}$ in (\ref{linsys}); cf.  (\ref{rhs}). 
In particular, $\|{\be}\|$ is available and we can apply the discrepancy principle with 
$\varepsilon=\|{\be}\|$ to determine the regularization parameter $\mu$ in Tikhonov 
regularization and the truncation index $k$ in TSVD. The parameter $\eta$ in 
(\ref{discrprinc}) and (\ref{discrprinc2}) is set to one. 

The computed approximation of $\hat{\bx}$ is denoted by ${\bx}_{\rm comp}$. We are 
interested in the relative error $\|{\bx}_{\rm comp}-\hat{\bx}\|/\|\hat{\bx}\|$ in the
computed solutions determined by Tikhonov regularization with the different regularization 
matrices described, and by TSVD. The difference ${\bx}_{\rm comp}-\hat{\bx}$ depends on 
the entries of the error vector ${\be}$. We report for every example the average of the 
relative errors in ${\bx}_{\rm comp}$ over $1000$ runs for each noise level. 

\begin{table}[htb!]
\centering
\begin{tabular}{cccccc}\hline
Noise level & \multicolumn{3}{c}{Tikhonov regularization} & TSVD\\
\% & $L$ in (\ref{Lnew}) & $L=\mu I$ &  $L$ in (\ref{Lnew2})\\
\hline
$10.0$ & $6.70 \cdot 10^{-2} $ & $6.83 \cdot 10^{-2}$ & 
${\bf 6.32} \cdot 10^{-2}$&$7.86\cdot 10^{-2}$\\
$\phantom{1}1.0$ & $2.72 \cdot 10^{-2} $ & $2.62 \cdot 10^{-2} $ & 
  $2.62 \cdot 10^{-2}$&${\bf 2.57} \cdot 10^{-2}$\\
$\phantom{1}0.5$ & $2.17 \cdot 10^{-2} $ & $2.08 \cdot 10^{-2} $ & 
  ${\bf 2.07} \cdot 10^{-2}$&$2.47 \cdot 10^{-2}$\\
  $\phantom{1}0.1$ & $1.08 \cdot 10^{-2} $ & $1.11 \cdot 10^{-2}$ & 
  ${\bf 1.03} \cdot 10^{-2}$&$1.23 \cdot 10^{-2}$\\

\hline
\end{tabular}
\caption{Example 3.1: Average relative errors in the computed solutions for the 
{\sf phillips} test problem for several noise levels.}\label{table:phillips}
\end{table}

Example 3.1. We first consider the problem {\sf phillips} from \cite{Ha2}. Let
\[
  \phi(t) = \left\{\begin{array}{lr}
              1 + \cos(\frac{\pi t}{3}), & |t| < 3, \\
             0,                          & |t| \geq 3,
             \end{array} \right. 
\]
and $a=c=-6$, $b=d=6$. The kernel, right-hand side function, and solution of the integral 
equation (\ref{fredholm}) are given by 
\[
    h(s,t) = \phi(s - t),~~
    x(t) = \phi(t),~~
    g(s)  = (6 - |s|)\left(1 + \frac{1}{2}
    \cos\left(\frac{\pi s}{3}\right)\right) + 
    \frac{9}{2 \pi} \sin\left(\frac{\pi |s|}{3}\right).
\]
Table \ref{table:phillips} displays the averages of the relative errors in the computed 
solutions over $1000$ runs for each noise level. The smallest average relative error is 
for each noise level marked in boldface. Tikhonov regularization with the regularization
matrix (\ref{Lnew2}) is seen to yield the same or smaller average errors as Tikhonov
regularization with the regularization matrices (\ref{Lnew}) and $\mu I$. The only average
error that is smaller than for Tikhonov regularization with the matrix (\ref{Lnew2}) is
obtained for $1\%$ noise by the TSVD method. We conclude that the regularization matrix 
(\ref{Lnew2}) yields competitive results and, in particular, determines more accurate
approximations of $\hat{\bx}$ than standard Tikhonov regularization (\ref{tikhonov}).
$\Box$

\begin{table}[htb!]
\centering
\begin{tabular}{cccccc}\hline
Noise level & \multicolumn{3}{c}{Tikhonov regularization} & TSVD\\
\% & $L$ in (\ref{Lnew}) & $L=\mu I$ &  $L$ in (\ref{Lnew2})\\
\hline
$10.0$ & ${\bf  1.69} \cdot 10^{-1} $ & $1.76 \cdot 10^{-1}$ & $1.70 \cdot 10^{-1}$&$1.86\cdot 10^{-1}$\\
$\phantom{1}1.0$ & ${\bf  1.02} \cdot 10^{-1} $ & $1.13 \cdot 10^{-1} $ & 
  $1.11 \cdot 10^{-1}$&$1.30\cdot 10^{-1}$\\
$\phantom{1}0.5$ & ${\bf  6.76} \cdot 10^{-2} $ & $8.35 \cdot 10^{-2} $ & 
  $7.53 \cdot 10^{-2}$&$7.86 \cdot 10^{-2}$\\
  $\phantom{1}0.1$ & $4.83 \cdot 10^{-2} $ & $5.03 \cdot 10^{-2}  $ & 
  ${\bf  4.80} \cdot 10^{-2}$&$4.83\cdot 10^{-2}$\\
\hline
\end{tabular}
\caption{Example 3.2: Average relative errors in the computed solutions for the
{\sf shaw} test problem for several noise levels.}\label{table:shaw}
\end{table}

Example 3.2. The test problem {\sf shaw} from \cite{Ha2} is an integral equation 
(\ref{fredholm}) with kernel and solution
\begin{eqnarray*}
h(s,t)&=& (\cos(s) + \cos(t))^2 \left(\frac{\sin(u)}{u}\right)^2, \quad
    u = \pi (\sin(s) + \sin(t)), \\
x(t)&=& 2\exp\left(-6\left( t-\frac{4}{5}\right)^2\right) + 
\exp\left(-2 \left(t + \frac{1}{2}\right)^2\right),
\end{eqnarray*}
and parameters $a=c=-\pi/2$, $b=d=\pi/2$. Table \ref{table:shaw} is analogous to Table 
\ref{table:phillips}; it displays the averages of the relative errors in the computed solutions over 
$1000$ runs for each noise level. The regularization parameter $\mu$ for Tikhonov 
regularization and the truncation index $k$ for TSVD are determined with the aid of the 
discrepancy principle. The smallest entry in each row is in boldface. The regularization
matrices (\ref{Lnew}) and (\ref{Lnew2}) can be seen to perform the best.

Table \ref{table:optimal} compares the performance of the methods when the optimal values
of the regularization parameter $\mu$ in Tikhonov regularization is used, i.e., we use the
values that give the most accurate approximations of $\hat{\bx}$. These values of $\mu$ 
are generally not available when solving discrete ill-posed problems. Nevertheless, it is 
interesting to see how the regularization matrices would perform if the optimal values of 
$\mu$ were available. The table shows, in increasing order, the average relative errors 
over $1000$ runs in the computed approximate solutions determined by Tikhonov 
regularization for the noise level $0.1\%$. All the modifications (\ref{Lnew}), 
(\ref{Lnew2}), and (\ref{Lnew5}) give approximate solutions of higher quality than 
$L=\mu I$. For the sake of completeness, we also report the average of the relative errors
in the computed solutions obtained with TSVD when the truncation index $k$ is chosen to give
the most accurate approximation of $\hat{\bx}$. It is $4.4777146 \cdot 10^{-2}$, which is
slightly larger than the average errors reported in Table \ref{table:optimal}. $\Box$

\begin{table}[htb!]
\centering
\begin{tabular}{ccccc}\hline
$L$ in (\ref{Lnew2})&  $L$ in (\ref{Lnew5})&$L$ in (\ref{Lnew})  & $L=\mu I$\\
\hline
$4.3750446\cdot 10^{-2}$&$4.3750452\cdot 10^{-2}$ & $4.3855830 \cdot 10^{-2} $ & $4.4713012\cdot 10^{-2}$\\ 
\hline
\end{tabular}
\caption{Example 3.2: Average relative errors in the computed solutions for the {\sf shaw}
test problem for noise level $0.1\%$ with optimal regularization parameters $\mu$ and $k$.
}\label{table:optimal}
\end{table}

Example 3.3. Consider the problem {\sf heat} from \cite{Ha2}. It is a discretization of a 
Volterra integral equation of the first kind on the interval $[0, 1]$ with a convolution 
kernel. Table \ref{table:heat} shows the average relative errors in the computed solutions
determined by Tikhonov regularization and TSVD over $1000$ runs for each noise level. The
regularization matrices (\ref{Lnew}) and (\ref{Lnew2}) are seen to yield the smallest 
average relative errors.  $\Box$

\begin{table}[htb!]
\centering
\begin{tabular}{cccccc}\hline
Noise level & \multicolumn{3}{c}{Tikhonov regularization} & TSVD\\
\% & $L$ in (\ref{Lnew}) & $L=\mu I$ &  $L$ in (\ref{Lnew2})\\
\hline
$10.0$ & $ 2.61 \cdot 10^{-1} $ & $2.88 \cdot 10^{-1}$ & $ {\bf 2.59} \cdot 10^{-1}$&$3.04\cdot 10^{-1}$\\
$\phantom{1}1.0$ & $ 9.95\cdot 10^{-2} $ & $1.08 \cdot 10^{-1} $ & 
  ${\bf 9.78}\cdot 10^{-2}$&$1.20\cdot 10^{-1}$\\
$\phantom{1}0.5$ & ${\bf 7.17} \cdot 10^{-2} $ & $7.75 \cdot 10^{-2} $ & 
  $7.21 \cdot 10^{-2}$&$9.67 \cdot 10^{-2}$\\
  $\phantom{1}0.1$ & $ 3.50 \cdot 10^{-2} $ & $3.67 \cdot 10^{-2}  $ & 
  $ {\bf 3.43} \cdot 10^{-2}$&$4.61\cdot 10^{-2}$\\
\hline
\end{tabular}
\caption{Example 3.3: Average relative errors in the computed solutions for the
{\sf heat} test problem for several noise levels.}\label{table:heat}
\end{table}

\subsection{Tests with colored noise}
In this subsection, we consider noise whose power density increases with the frequency, 
i.e., the noise has more energy in the high frequencies than white Gaussian noise. This 
kind of noise is known as ``colored noise'' and is sometimes referred to as ``violet noise'';
see, e.g., Hansen \cite{Ha3} for a discussion of colored noise in discrete ill-posed problems.
Let $U$ be the orthogonal matrix of left singular vectors of the matrix $A$ in 
(\ref{linsys}).  Hansen \cite[p.~74]{Ha3} generates colored noise with the MATLAB command
\begin{equation}\label{colnoise}
\texttt{e=U*(logspace(-alpha,0,200)'.*(U'*randn(200,1)));}
\end{equation}
Here \texttt{randn(200,1)} yields a vector in $\R^{200}$ with normally distributed random
entries and the parameter $\alpha=\texttt{alpha}$ determines how much the energy in the
high frequencies dominate; they dominate more the larger $\alpha>0$. We add the vector
$\be=\texttt{e}$ to the noise-free data vector $\hat{\bb}$ to obtain the 
noise-contaminated data vector $\bb$; cf. (\ref{rhs}). When the covariance matrix for the
noise is known, then its Cholesky factorization can be used to prewhitening the noise; see
\cite[p.~76]{Ha3}. We assume the covariance matrix not to be available and would 
like to illustrate how the methods considered in this paper perform in this situation. The
vector ${\be}$ is scaled to yield a specified noise level $\|{\be}\|/\|\hat{\bb}\|$ and we
use the discrepancy principle to determine the regularization parameters in Tikhonov 
regularization and TSVD with $\eta=1$ in (\ref{discrprinc}) and (\ref{discrprinc2}).
We also will replace the matrix $U$ in (\ref{colnoise}) by other orthogonal matrices. 

Example 3.4. Consider the integral equation of the first kind (\ref{fredholm}) with the 
kernel and right-hand side function given by
\begin{eqnarray*}
 h(s,t) &=& \left\{\begin{array}{lr}
 s \, (t-1), & \ s<t, \\
 t \, (s-1), & \ s\geq t,
 \end{array}
 \right. 
\end{eqnarray*}
and
\begin{eqnarray*} 
g(s) &=& \left\{\begin{array}{lr}
 (4\, s^3 - 3\, s)/24,  &  s <  0.5, \\
 (-4\, s^3 + 12\, s^2 - 9\, s + 1)/24, &  s \geq  0.5. \\
 \end{array}
 \right.
\end{eqnarray*}
We use the MATLAB function {\sf deriv2} from \cite{Ha2} to determine a discretization
$A\in \R^{200 \times 200}$ of the integral operator, and a scaled discrete approximation
$\hat{\bx}$ of the solution 
\begin{eqnarray*} 
x(t) &=& \left\{\begin{array}{lr}
 t,  &  t <  0.5, \\
1-t, &  t \geq  0.5. \\
\end{array}
\right.
\end{eqnarray*}
We compute the noise-free data vector $\hat{\bb}:=A\hat{\bx}$ to which we add the 
noise-vector $\be$. The latter is generated by (\ref{colnoise}) with $\alpha=1$ followed 
by scaling. 

Table \ref{table:deriv2} displays the averages of the relative errors in the computed 
solutions over $1000$ runs for each noise level. Tikhonov regularization with the 
regularization matrix (\ref{Lnew2}) is seen to yield the smallest average errors for all 
noise levels. Table \ref{table:deriv2b} is obtained by replacing the orthogonal matrix $U$
of left singular vectors in (\ref{colnoise}) by a random orthogonal matrix, and for the
results of Table \ref{table:deriv2c} this matrix is replaced by the orthogonal cosine
transform matrix. The regularization matrix (\ref{Lnew2}) is seen to perform well in 
each one of these tables. $\Box$

\begin{table}[htb!]
\centering
\begin{tabular}{cccccc}\hline
Noise level & \multicolumn{3}{c}{Tikhonov regularization} & TSVD\\
\% & $L$ in (\ref{Lnew}) & $L=\mu I$ &  $L$ in (\ref{Lnew2})\\
\hline
$\phantom{1}1.0$ & ${2.31} \cdot 10^{-2} $ & $2.25 \cdot 10^{-2} $ & 
  ${\bf 2.16} \cdot 10^{-1}$&$2.34\cdot 10^{-1}$\\
$\phantom{1}0.5$ & ${1.81} \cdot 10^{-2} $ & $1.76 \cdot 10^{-2} $ & 
  ${\bf  1.72} \cdot 10^{-2}$&$1.81\cdot 10^{-2}$\\
  $\phantom{1}0.1$ & $1.01\cdot 10^{-2} $ & $9.86 \cdot 10^{-3}  $ & 
  ${\bf 9.62} \cdot 10^{-3}$&$1.03\cdot 10^{-2}$\\
\hline
\end{tabular}
\caption{Example 3.4: Average relative errors in the computed solutions for the 
{\sf deriv2} test problem for several noise levels. Moderate violet noise 
($\alpha=1$).}\label{table:deriv2}
\end{table}

\begin{table}[htb!]
\centering
\begin{tabular}{cccccc}\hline
Noise level & \multicolumn{3}{c}{Tikhonov regularization} & TSVD\\
\% & $L$ in (\ref{Lnew}) & $L=\mu I$ &  $L$ in (\ref{Lnew2})\\
\hline
$\phantom{1}1.0$ & ${3.90} \cdot 10^{-2} $ & $3.78 \cdot 10^{-2} $ & 
  ${\bf 3.65} \cdot 10^{-1}$&$4.07\cdot 10^{-1}$\\
$\phantom{1}0.5$ & ${3.05} \cdot 10^{-2} $ & $2.96 \cdot 10^{-2} $ & 
  ${\bf  2.92} \cdot 10^{-2}$&$3.02\cdot 10^{-2}$\\
  $\phantom{1}0.1$ & $1.65\cdot 10^{-2} $ & $1.62 \cdot 10^{-3}  $ & 
  ${\bf 1.56} \cdot 10^{-3}$&$1.74\cdot 10^{-2}$\\
\hline
\end{tabular}
\caption{Example 3.4: Average relative errors in the computed solutions for the
{\sf deriv2} test problem for several noise levels. $U$ in (\ref{colnoise}) is an 
orthogonal random matrix. Violet noise ($\alpha=2$).}\label{table:deriv2b}
\end{table}

\begin{table}[htb!]
\centering
\begin{tabular}{cccccc}\hline
Noise level & \multicolumn{3}{c}{Tikhonov regularization} & TSVD\\
\% & $L$ in (\ref{Lnew}) & $L=\mu I$ &  $L$ in (\ref{Lnew2})\\
\hline
$\phantom{1}1.0$ & ${2.32} \cdot 10^{-2} $ & $2.25 \cdot 10^{-2} $ & 
  ${\bf 2.16} \cdot 10^{-1}$&$2.34\cdot 10^{-1}$\\
$\phantom{1}0.5$ & ${1.81} \cdot 10^{-2} $ & $1.76 \cdot 10^{-2} $ & 
  ${\bf  1.72} \cdot 10^{-2}$&$1.80\cdot 10^{-2}$\\
  $\phantom{1}0.1$ & $1.02\cdot 10^{-2} $ & $9.90 \cdot 10^{-3}  $ & 
  ${\bf 9.64} \cdot 10^{-3}$&$1.03\cdot 10^{-2}$\\
\hline
\end{tabular}
\caption{Example 3.4: Average relative errors in the computed solutions for the
{\sf deriv2} test problem for several noise levels. $U$ in (\ref{colnoise}) is a
orthogonal cosine transform matrix. Moderate violet noise ($\alpha=1$).}\label{table:deriv2c}
\end{table}

Example 3.5. Consider again the test problem {\sf heat} from \cite{Ha2}. Tables
\ref{table:heat_v1} and \ref{table:heat2} are analogous to Tables \ref{table:deriv2} and 
\ref{table:deriv2b}, respectively. Tikhonov regularization with the regularization matrix 
(\ref{Lnew2}) is seen to perform well.  $\Box$

\begin{table}[htb!]
\centering
\begin{tabular}{cccccc}\hline
Noise level & \multicolumn{3}{c}{Tikhonov regularization} & TSVD\\
\% & $L$ in (\ref{Lnew}) & $L=\mu I$ &  $L$ in (\ref{Lnew2})-(\ref{Lnew5})\\
\hline
$\phantom{1}1.0$ & ${5.78} \cdot 10^{-2} $ & $5.92 \cdot 10^{-2} $ & 
  ${\bf 5.40} \cdot 10^{-2}$&$6.76\cdot 10^{-2}$\\
$\phantom{1}0.5$ & ${4.34} \cdot 10^{-2} $ & $4.36 \cdot 10^{-2} $ & 
  ${\bf  4.21} \cdot 10^{-2}$&$4.95 \cdot 10^{-2}$\\
  $\phantom{1}0.1$ & $2.48 \cdot 10^{-2} $ & ${\bf 2.29} \cdot 10^{-2}  $ & 
  ${ 2.30} \cdot 10^{-2}$&$2.34\cdot 10^{-2}$\\
\hline
\end{tabular}
\caption{Example 3.5: Average relative errors in the computed solutions for the {\sf heat}
test problem for several noise levels. Moderate violet noise ($\alpha=1$).}
\label{table:heat_v1}
\end{table}

\begin{table}[htb!]
\centering
\begin{tabular}{cccccc}\hline
Noise level & \multicolumn{3}{c}{Tikhonov regularization} & TSVD\\
\% & $L$ in (\ref{Lnew}) & $L=\mu I$ &  $L$ in (\ref{Lnew2})-(\ref{Lnew5})\\
\hline 
$\phantom{1}1.0$ & ${9.76} \cdot 10^{-2} $ & ${  1.06} \cdot 10^{-1} $ & 
  ${\bf 9.67} \cdot 10^{-2}$&$1.18\cdot 10^{-1}$\\
$\phantom{1}0.5$ & ${\bf 7.14} \cdot 10^{-2} $ & $7.73 \cdot 10^{-1} $ & 
  ${  7.18} \cdot 10^{-2}$&$9.68\cdot 10^{-1}$\\
  $\phantom{1}0.1$ & ${ 3.50}\cdot 10^{-2} $ & $3.70 \cdot 10^{-2}  $ & 
  ${ \bf 3.44} \cdot 10^{-2}$&$4.61\cdot 10^{-2}$\\
\hline
\end{tabular}
\caption{Example 3.5: Average relative errors in the computed solutions for the {\sf heat} 
test problem for several noise levels. The matrix $U$ in (\ref{colnoise}) is an orthogonal
random matrix. Violet noise ($\alpha=2$).}\label{table:heat2}
\end{table}

\section{Conclusion and extension}\label{sec4}
Tikhonov regularization suggests several matrix nearness problems for determining
regularization matrices. Regularization matrices so defined can give approximate solutions
of higher quality than both Tikhonov regularization (\ref{tikhonov}) with regularization 
matrix $\mu I$ and the TSVD method. The computational effort is dominated by the 
computation of the SVD (\ref{svd}) of the given matrix $A$ in (\ref{linsys}) and, 
consequently, is essentially the same for all methods considered in this paper. The new 
regularization matrices are attractive both when the noise $\be$ is white Gaussian or 
violet.

For ease of description of the methods, we assumed the SVD of $A$ to be available. This
requirement can be removed. A least-squares problem (\ref{linsys}) with a matrix too 
large to compute its SVD can be reduced to small a problem by a Krylov subspace method.
The methods of the present paper can be applied to the reduced problem so obtained. 
Reduction methods include partial Golub--Kahan bidiagonalization and partial Arnoldi
decomposition; see, e.g., \cite{Bj,CGR,NRS,Sa} for illustrations of application of these
reduction methods.

We also note that the methods of this paper can be applied to Tikhonov regularization
problems (\ref{tikhonov}) with a more general regularization matrix than $\mu I$ by
first transforming the more general problem to the form (\ref{tikhonov}). Transformation
methods are discussed in \cite[Sections 2.3.1 and 2.3.2]{Ha1} and \cite{RY}.

We used the discrepancy principle to determine the amount of regularization in all
computed examples. However, the regularization methods described also can be applied in 
conjunction with parameter choice rules that do not require a bound for $\|{\be}\|$ to be
known. Many such parameter choice rules are discussed and analyzed in
\cite{BRRS,BRS,Ha1,Ha3,Ki,Ki2,RR} and in references therein.

\section*{Acknowledgement}
We would like to thank a referee for comments that improved the presentation.

\end{document}